\documentclass[12pt]{article}
\usepackage{amsfonts}
\usepackage{mathrsfs}
\usepackage{epsfig}
\usepackage{amssymb,amsmath,amsthm,amscd}
\usepackage{latexsym}
\usepackage{xcolor,bm}
\usepackage{xcolor,bm}
\usepackage{cite}
\usepackage{color} 
 \theoremstyle{theorem}
 \newtheorem{thm}{Theorem}[section]
 \newtheorem{lem}{Lemma}[section]
  
  \newtheorem{remark}{Remark}[section]
 \newtheorem{defn}{Definition}[section]
 \newtheorem{ex}{Example}[section]

\newcommand{\cA}{{\mathcal A}}

\newcommand{\cC}{{\mathcal C}}

\newcommand{\cK}{{\mathcal K}}

\newcommand{\cM}{{\mathcal M}}

\newcommand{\cS}{{\mathcal S}}

\newcommand{\cX}{{\mathcal X}}

\newcommand{\mB}{{\mbox B}}

\def\R{\mathbb{R}}

\def\1{\mathbb{1}}

\def\bc{\begin{center}}
\def\ec{\end{center}}
\def\be{\begin{equation}}
\def\ee{\end{equation}}
\def\ba{\begin{array}}
\def\ea{\end{array}}
\def\benu{\begin{enumerate}}
\def\eenu{\end{enumerate}}
\def\bt{\begin{theorem}}
\def\et{\end{theorem}}
\def\bl{\begin{lemma}}
\def\el{\end{lemma}}
\def\bco{\begin{corollary}}
\def\eco{\end{corollary}}
\def\br{\begin{remark}}
\def\er{\end{remark}}
\def\bd{\begin{definition}}
\def\ed{\end{definition}}
\def\bp{\begin{proposition}}
\def\ep{\end{proposition}}
\def\bo{\begin{proof}}
\def\eo{\end{proof}}
\def\bx{\begin{example}}
\def\ex{\end{example}}

\def\pa{\partial}
\def\a{\alpha}
 \def\de{\delta}

\def\lam{\lambda} \def\Lam{\Lambda}

\def\ve{\varepsilon}
\def\sig{\sigma}

\def\w{\omega}
\def\gam{\gamma}\def\Gam{\Gamma}\def\e{{\rm e}}

\def\~{\widetilde}
\def\A{\forall}
\def\ol{\overline}
\def\ul{\underline}
\def\Cap{\bigcap}
\def\Cup{\bigcup}
\def\ra{\rightarrow}

\def\8{\infty}
\def\X{\times}

\def\mb{\mbox}

\def\ss{\subset}

\def\.{\cdot}
\def\leq{\leqslant}\def\geq{\geqslant}\def\d{{\rm d}}

\def\Hs{\hspace{0.8cm}}
\def\hs{\hspace{0.4cm}}
\def\Vs{\vskip10pt}
\def\vs{\vskip5pt}

\def\[{\left[}
\def\]{\right]}
\def\({\left(}
\def\){\right)}

\title{Bifurcation from Infinity of the Schr\"odinger Equation via Invariant Manifolds
  }

\author{
\small{Chunqiu Li} \footnote{E-mail:
lichunqiu@wzu.edu.cn},\Hs 
{Jintao Wang} \footnote{Corresponding author. E-mail: wangjt@wzu.edu.cn}  ~\\\\
\small\it Department of Mathematics, Wenzhou University,\\ \small\it Wenzhou, Zhejiang Province, 325035, P. R. China\\
}
\date{\small\today}

\begin{document}

\maketitle
\begin{abstract}\baselineskip 14pt
This paper is concerned with the bifurcation from infinity of the nonlinear Schr\"odinger equation $$-\Delta u+V(x)u=\lam u+f(x,u),\hs x\in \mathbb{R}^N.$$ We treat this problem in the framework of dynamical systems by considering the corresponding parabolic equation on unbounded domains. Firstly, we establish a global invariant manifold for the parabolic equation on $\R^N$. Then, we restrict the parabolic equation to this invariant manifold, which generates a system of finite dimension. Finally, we use the Conley index theory and the shape theory of attractors to establish some new results on bifurcations from infinity and multiplicity of solutions 
of the Schr\"odinger equation 
under an appropriate Landesman-Lazer type condition. 
\end{abstract}


\setcounter {equation}{0}
\section{Introduction}

In this paper we are concerned with the following nonlinear Schr\"odinger equation
\be\label{e1.1}
-\Delta u+V(x)u=\lam u+f(x,u),\hs x\in \mathbb{R}^N,
\ee
where the potential $V\in L^\infty(\R^N)$, $N\geq 1$, $\lam\in \R$ is the bifurcation parameter, and $f$ is bounded and satisfies the following Landesman-Lazer type condition
\be\label{1.2}
  \liminf\limits_{s\rightarrow+\infty}f(x,s)\geq\ol{f}>0,\Hs
  \limsup\limits_{s\rightarrow-\infty}f(x,s)\leq-\underline{f}<0
  \ee uniformly with respect to $x\in\R^N$ (where $\ol{f}$ and $\underline{f}$ are independent of $x$). We are basically interested in the bifurcation from infinity and multiplicity of solutions of the equation.

The bifurcation and multiplicity of elliptic equations near resonance have aroused much interest in the past decades. This topic can be traced back to the earlier work of  Mawhin and Schmitt \cite{MS}, where the authors studied the problem of elliptic equations of the following form
\be\label{1.3}
-\Delta u=\lam u+f(x,u)
\ee
on a bounded domain, where $\lam\in \R$ and $f$ satisfies an appropriate Landesman-Lazer type condition. It was shown that if $\mu_k$ is an eigenvalue of odd multiplicity, then \eqref{1.3} has at least two distinct solutions for $\lam$ on one side of $\mu_k$ but close to $\mu_k$, and at least one solution for $\lam$ on the other side. Later Schmitt and Wang \cite{SW} extended this result to the case when $\lam$
crosses each eigenvalue $\mu_k$ on
bifurcations from infinity for operator equations to cover the case of even multiplicity. Recently, Li, Li and Zhang \cite{LLZ} presented a dynamical argument for Schmitt and Wang's result on \eqref{1.3}. The interested reader is referred to \cite{BLL,CL,CMN,DO,FGP,LW,MS2,PM,PP,ST,TW}, etc., on bifurcations from infinity and multiplicity for elliptic equations under various boundary conditions.



The bifurcation from infinity of the equation \eqref{e1.1} on unbounded domains is also widely investigated; see \cite{CK,KS,G,G1,Stu,SZ} and the references
therein. However, comparing to the case of bounded domains, the problem on unbounded domains is more difficult. The main reason is that the spectrum of the operator $A:=-\Delta+V(x)$ is not discrete in general, and may be quite complicated, which depends on the potential. Under some additional conditions on the potential $V$ and $f$, Kryszewski and Szulkin \cite{KS}, using the degree theory and a variational approach, considered the asymptotic linear Schr\"odinger equation \eqref{e1.1}. They showed that if $\lam_0$ is an isolated eigenvalue for the linearization at infinity, then there exists a sequence $(u_n,\lam_n)$ of solutions of \eqref{e1.1} such that
$$\|u_n\|\ra\8\hs \mb{and \hs $\lam_n\ra\lam_0$,\Hs as \hs $n\ra\8$,}$$ which extends the results of Stuart \cite{Stu}.
Later, by using degree theory and Conley index theory, \'Cwiszewski and Kryszewski \cite{CK} studied the bifurcation from infinity of elliptic equations on $\R^N$, and  proved that if the bifurcation parameter is an eigenvalue of the Hamiltonian, then the bifurcation from infinity also occurs under some suitable conditions. Motivated by these works mentioned above, in this paper we further study the bifurcation from infinity and multiplicity of the Schr\"oding equation \eqref{e1.1} near an isolated eigenvalue $\lam_*$ of the operator $A=\Delta+V(x)$. We will show that under some appropriate conditions on $f$ and the potential $V$, there exists $\theta>0$ such that for each $\lam\in [\lam_*-\theta,\lam_*)$, \eqref{e1.1} has at least three distinct solutions $e_\lam^1,e_\lam^2$ and $e_\lam^3$ with
\be\label{ee}
\|e_\lam^i\|\ra \8,\hs \mb{as}\hs \lam\ra\lam_*^{-},\hs i=1,2,
\ee
whereas $e_\lam^3$ remains bounded on $[\lam_*-\theta,\lam_*)$. The ``dual" version of our results also holds true if, instead of \eqref{1.2}, we assume that
\be\label{1.5}
  \limsup\limits_{s\rightarrow+\infty}f(x,s)\leq-\ol{f}<0,\Hs
  \liminf\limits_{s\rightarrow-\infty}f(x,s)\geq\underline{f}>0
  \ee
uniformly with respect to $x\in\R^N$.

Our method in this work is arranged as follows. Instead of transforming \eqref{e1.1} into an operator equation and using the topological degree or variational methods, we regard the problem as the stationary one of the following parabolic equation
\be\label{e1.4}
u_t-\Delta u+V(x)u=\lam u+f(x,u),\hs x\in \R^N.
\ee
Specifically, let $\Phi_\lam$ be the semiflow generated by \eqref{e1.4} and
$\lam_*$ be an isolated eigenvalue of the operator $A=-\Delta+V(x)$. Assume that the Lipschitz constant $L_f$ of $f$ satisfies
$$F_{\mu}L_f<1$$
for $\mu\in (\frac{1}{4}\beta,\frac{3}{4}\beta)$ fixed, where $F_\mu$ and $\beta>0$ are given by \eqref{L} and \eqref{beta}, respectively.
We first present an existence result on a global invariant manifold $\cM$ of the equation \eqref{e1.4} for $\lam$ near $\lam_*$. 
Then the equation \eqref{e1.4} can be restricted to this invariant manifold $\cM$, and therefore it is reduced to a system of finite dimension. This allows us to apply the Conley index theory to the reduced system. By virtue of the shape theory of attractors \cite{KR}, it can be shown that there exists $\theta>0$ such that for each $\lam\in [\lam_*-\theta,\lam_*)$, the system \eqref{e1.4} can bifurcate from infinity a compact invariant set $\cS_\lam^\8$, which takes the shape of a sphere $\mathbb{S}^{m-1}$, where $m$ denotes the algebraic multiplicity of $\lam_*$. Based on this fact, one can further verify that $\cS_\lam^\8$ contains at least two distinct equilibria $e_\lam^1$ and $e_\lam^2$ of $\Phi_\lam$, which are actually solutions of \eqref{e1.1} satisfying \eqref{ee}.

Besides, using the Landesman-Lazer type condition \eqref{1.2}, we show a fundamental result on $f$ (see Lemma \ref{l4.1}), which plays a crucial role in establishing our main results for \eqref{e1.1}. Then we prove that \eqref{e1.1} has  a solution $e_\lam^3$ for $\lam$ belonging to a two-sided neighborhood $\Lam$ of $\lam_*$,
which remains bounded on $\Lam$.
We remark that this lemma extends the result established in \cite{LLZ,LSS} (see \cite[Lemma 5.2]{LLZ} or \cite[Lemma 6.7]{LSS}) to the case of unbounded domains, and can be seen as a nontrivial extension. This is because it is very hard to verify this result by directly using the techniques given in \cite{LLZ,LSS}. To avoid the difficulty caused by unbounded domains, we establish this lemma on some finite dimensional subspace of the phase space, which is inspired by the finite-dimension property of the invariant manifold.

It is worth mentioning that we treat the problem \eqref{e1.1} in the framework of dynamical systems, and deal with it via the approach of a pure dynamical nature, which is different from those in the literature. Moreover, the method that we restrict the system \eqref{e1.4} to the global invariant manifold can help us avoid the verification of the compactness for the semiflow $\Phi_\lam$. In general, it is not easy to verify that the semiflow is asymptotically compact for a parabolic equation on unbounded domains.


This work is organized as follows. Section 2 is concerned with some
preliminaries. In Section 3 we first establish a global invariant manifold of the equation \eqref{e1.4} on $\R^N$, and then we reduce \eqref{e1.4} to this invariant manifold. In Section 4, we present a detailed discussion on dynamical behaviors of the reduced system, and establish our main results on bifurcations from infinity and multiplicity of solutions of the Schr\"odinger equation \eqref{e1.1}.

\setcounter {equation}{0}
\section{Preliminaries}

In this section we first make some preliminaries.
\subsection{Local semiflows}
Let $X$ be a complete metric space.

A {\em local semiflow} $\Phi$ on $X$ is a continuous mapping from an
open set $\mathcal{D}(\Phi)\subset \mathbb{R}_+\times X$ to $X$ that
satisfies the following properties \benu
\item[{\rm(1)}] For every $x\in X$, there is $0<T_x\leq\infty$, such that
$$(t,x)\in\mathcal{D}(\Phi)\Longleftrightarrow t\in [0,T_x);$$
\item[{\rm(2)}] $\Phi(0,\cdot)={\rm id}_X$, and
$$\Phi(t+s,x)=\Phi(t,\Phi(s,x))$$
for any $x\in X$ and $t,s\in\R_+$ with $t+s\leq T_x$.
 \eenu
Here, the number $T_x$ is called the {\em escape time} of $\Phi(t,x)$. For simplicity, we rewrite $\Phi(t,x)$ as $\Phi(t)x$.

Given an interval $J\subset\R$. A continuous mapping $\gam:J\ra X$ is called a {\it trajectory} (or {\em solution}) of $\Phi$ on $J$
if
$$
  \gam(t)=\Phi(t-s)\gam(s),\hs \forall t,s\in J,\, t\geq s.
  $$
We call a trajectory $\gam$ on $\R$ a {\it full trajectory}. 
The orbit of a full trajectory $\gam$ is the set $${\rm orb}(\gam)=\{\gam(t):t\in\R\},$$ which is simply called {\it a full orbit}.

The {\em $\omega$-limit set} $\w(\gam)$ and {\em $\w^*$-limit set} $\w^*(\gam)$ of a full trajectory $\gam$ are defined, respectively, by
$$\ba{ll}
\w(\gamma)=\{y\in X:\,\,\,\,\mb{there exists }  t_n\ra \8 \mb{ such that }\gamma(t_n)\ra y\},\ea$$
$$
\ba{ll}
\w^*(\gamma)=\{y\in X:\,\,\,\,\mb{there exists }  t_n\ra -\8 \mb{ such that }\gamma(t_n)\ra y\}.\ea
$$

Let $N$ be a subset of $X$. We say that $\Phi$ {\em does not explode} in $N$, if $\Phi([0,T_x))x\subset N$ implies that  $T_x=\infty$.

\begin{defn}\label{defn2.3}\cite{Ry}\,
A set $N\subset X$ is said to be admissible, if for every sequences $x_n\in
N$ and $t_n\rightarrow\infty$ satisfying $\Phi([0,t_n])x_n\subset N$ for
all $n$, then the sequence $\Phi(t_n)x_n$ has a convergent subsequence.

A set $N$ is said to be strongly admissible, if it is admissible and
moreover, $\Phi$ does not explode in $N$.
\end{defn}
\begin{defn}\label{defn2.4}
$\Phi$ is said to be asymptotically compact  on $X$, if every bounded
set $B\subset X$ is strongly admissible.
\end{defn}
A set $S\ss X$ is said to be {\em positively invariant} (resp., {\em invariant}), if $\Phi(t)S\subset S$ (resp., $\Phi(t)S=S$) for any $t\geq0$.

A compact invariant set $\cA\ss X$ is called an {\em attractor} of  $\Phi$, if it attracts a neighborhood $U$ of itself, that is,
$$
\lim_{t\ra\8}\d_H\(\Phi(t)U,\cA\)=0.
$$
\subsection{Sectorial operators and spectral sets}

Let $X$ be a Banach space. A closed densely defined linear operator $A: D(A)\subset X \ra X$ is called a {\it sectorial operator}, if there are real numbers $\phi\in(0,\pi/2), a\in \mathbb{R}$ and $M\geq 1$ such that the sector
$$
  S_{a,\phi}=\{\lam: \phi\leq |{\rm arg}(\lam-a)|\leq \pi,\hs \lam\neq a\}
$$
is contained in the resolvent set of $A,$ and
$$
  \|(\lam I-A)^{-1}\|\leq M/|\lam-a|
  $$
for each $\lam\in S_{a,\phi}$, where $I$ denotes the identity on $X.$

Let $A$ be a sectorial operator on $X$. By $\sig(A)$ we denote the spectrum of $A.$ If $\min_{z\in \sig(A)}{\rm Re}\,z>0,$ then
one can define the fractional powers of $A$ as follows: for each $\a>0,$
$$
  A^{-\a}=\frac{1}{\Gam(\a)}\int^\infty_0t^{\a-1}\e^{-At}\d t,
$$
where $\ba{l}\Gam(s)=\int^\infty_0t^{s-1}\e^{-t}\d t\ea$ is the Gamma function. Let $A^\a$ be the inverse of $A^{-\a}$ with $D(A^\a)=R(A^{-\a}).$ 

Let $\a\geq 0$. Denote $X^\a=D(A_1^\a)$ and equip $X^\a$ with the norm $\|\cdot\|_\a$ defined by
$$
  \|u\|_\a=\|A^\a_1 u\|,\hs u\in X^\a,
$$
where $A_1=A+aI$, and $a$ is a real number such that $\min_{z\in \sig(A_1)}{\rm Re}z>0$.
It is easy to see that $X^\a$ is a Banach space, which is called the {\it fractional power of $X$}. For more notions and results on sectorial operators, the interested reader is referred to Henry \cite[Chapter I]{Hen} for details.




\vs
Now let us briefly recall some notions on spectral sets; see \cite{KS} and \cite{S}, etc., for details.

Let $X,Y$ be (real) Banach spaces. A closed densely defined linear operator $L:D(L)\subset X\ra Y$ is called a {\it Fredholm operator} if the range $R(L)$ is closed and $\dim N(L)<\infty, {\rm codim}\, R(L)<\infty,$ where $N(L)$ is the kernel of $L$.

\begin{defn}\label{2.1}
Let $E$ be a real Hilbert space and $L:D(L)\subset E \ra E$ be a self-adjoint linear operator. The essential spectrum $\sigma_e(L)$ of $L$ is defined by the set
$$
  \{\lam\in \mathbb{C}:L-\lam I\,\,\mb{ is not a Fredholm operator}\},
$$
\end{defn}
\begin{remark}\label{r2.1}
One may immediately deduce from the definition that $\sigma_e(L)\subset \sigma(L)$ and that $\sig(L)\setminus \sig_e(L)$ consists of isolated eigenvalues of finite multiplicity;  see also \cite{KS}.
\end{remark}

\subsection{Conley index}
For the readers' convenience, we finally recall the definition of Conley index; see
\cite{Con, Mis} and \cite{Ry}, etc., for details.

Let $\Phi$ be a local semiflow on $X$. From now on we will always assume $\Phi$ is asymptotically compact.

A compact invariant set $S$ is said to be {\em isolated}, if
there exists a neighborhood $N$ of $S$ such that $S$
is the maximal compact invariant set of $\Phi$ in $\ol N$.
Correspondingly, the set $N$ is called an
{\em isolating neighborhood} of $S$.



Let $B$ be a bounded closed subset of $X$. $x\in \pa B$ is called a {\em strict ingress} (resp., {\em strict egress}, {\em bounce-off}) point of $B$, if for each trajectory $\gamma:[-\tau,s]\ra X$ with $\gamma(0)=x$, where $\tau\geq0$, $s>0$, the following conditions are satisfied:
\vs
\noindent(1) there exists  $0<\ve<s$ such that
$$
\gamma(t)\in \mb{int}B\,\,\,(\mb{resp., }\, \gamma(t)\not\in B,\,\,\,\mb{resp., }\,\gamma(t)\not\in B),\Hs \A\,t\in (0,\ve);
$$
(2) if $\tau>0$, then there is a number $\de \in (0,\tau)$ such that
$$
\gamma(t)\not\in B \,\,\,(\mb{resp., }\, \gamma(t)\in \mb{int}B,\,\,\,\mb{resp., }\,\gamma(t)\not\in B),\Hs \A\,t\in (-\de, 0).
$$
By $B^i$ (resp., $B^e$, $B^b$), we denote the set of all strict ingress (resp., strict egress, bounce-off) points of the closed set $B$, and write $B^-=B^e\cup B^b$.

A set $B\subset X$ is called an {\em isolating block} \cite{Ry} if $B^-$ is closed and $
\pa B=B^i\cup B^-.$
\vs
Let $N,E$ be two closed subsets of $X$.  $E$ is called an {\em exit
set} of $N$, if the following properties hold:
\benu
\item[(1)] $E$ is {\em N-positively invariant}, that is, for each $x\in E$ and $t\geq 0$,
$$\Phi([0,t])x\subset N \Longrightarrow\Phi([0,t])x\subset E;$$
\item[(2)] for every $x\in N,$ if $\Phi(t_1)x\not\in N$ for some
$t_1>0$, then there is $t_0\in[0,t_1]$ such that $\Phi(t_0)\in E$.
\eenu

Let $S$ be a compact isolated invariant set of $\Phi$. A pair of bounded closed
subsets $(N,E)$ is said to be an {\em index pair} of $S$, if (1)\, $N\setminus E$ is an isolating neighborhood of
$S;$ (2)\, $E$ is an exit set of $N$.

By \cite{Ry}, one can deduce that if $B$ is a bounded isolating block, then $(B,B^-)$ is an index pair of the maximal compact invariant set $S$  in $B$.

\begin{defn}\label{defn2.8}
Let $(N,E)$ be an index pair of $S$. The homotopy Conley index
of $S$ is defined to be the homotopy type $[(N/E,[E])]$ of the
pointed space $(N/E,[E])$, denoted by $h(\Phi,S).$
\end{defn}
\setcounter {equation}{0}
\section{Invariant manifolds of nonlinear evolution equations}

In this section, we establish an existence result on global invariant manifolds of the following nonlinear evolution equation
\be\label{e3.1}
u_t+Au=\lam u+f(u)
\ee
on a Banach space $X$, where $A$ is a sectorial operator on $X$, $\lam\in \R$, and $f$ is a globally Lipschitz mapping from $X^\a$ to $X$ for some $0\leq \a<1$.

Let $\|\cdot\|$ and $\|\cdot\|_\a$ denote the norms of $X$ and $X^\a$, respectively.
\subsection{Prelimiaries}

Denote $\sig(A)$ the spectrum of $A$ and
$${\rm Re}\,\sig(A):=\{{\rm Re}\,z: z\in \sig(A)\}.$$
Suppose $\lam_0$ is an isolated eigenvalue of $A$ with finite multiplicity, and that the spectrum $\sigma(A)$ has a decomposition $\sig(A)=\sig^1\cup\sig^2\cup\sig^3$ with
$$
  \sig^1=\sig(A)\cap\{{\rm Re}\lam\leq \beta_1\},\hs \sig^2=\{\lam_0\},\hs \sig^3=\sig(A)\cap\{{\rm Re}\lam\geq \beta_2\},
$$ where $\beta_1,\beta_2$ are real numbers. It can be easily seen that
$$
  {\rm Re}\,\sig(A)\cap(\beta_1,\beta_2)=\{{\rm Re}\lam_0\}.
$$
Correspondingly, the space $X$ can be decomposed as
$$
  X=X^1\oplus X^2 \oplus X^3,
$$
and the space $X^2$ is finite dimensional. In addition, we assume the space $X^1$ is also finite dimensional. Denote
$$P_i: X \ra X^i,\, i\in\{1,2,3\}$$ the projection from $X$ to $X_i$.

Pick a positive number $\beta$ with
\be\label{beta}
  \beta<\min\{{\rm Re}\lam_0-\beta_1,\beta_2-{\rm Re}\lam_0\},
  \ee
and write $J=({\rm Re}\lam_0-\frac{1}{4}\beta,{\rm Re}\lam_0+\frac{1}{4}\beta).$ Let $A_\lam=A-\lam I, \lam\in J$ and denote $A^i=A_\lam|_{X^i}$. Then by the basic knowledge on sectorial operators (see e.g., \cite[Theorems 1.5.3, 1.5.4]{Hen}), one deduces that there exists $M>0$ (depending on $A$) such that for each $\a\in[0,1)$,
\begin{align}
  &\|\Lam^\a {\e}^{-A^1 t}\|\leq M\e^{\frac{3}{4}\beta t},\hs \|\e^{-A^1 t}\|\leq M\e^{\frac{3}{4}\beta t},\Hs t\leq 0,\label{e3.2}\\
  &\|\Lam^\a \e^{-A^2 t}\|\leq M\e^{\frac{1}{4}\beta |t|},\hs \|\e^{-A^2 t}\|\leq M\e^{\frac{1}{4}\beta |t|},\Hs t\in \R,\label{e3.3}\\
  &\|\Lam^\a \e^{-A^3 t}P_3\Lam^{-\a}\|\leq M\e^{-\frac{3}{4}\beta t},\hs \|\Lam^\a \e^{-A^3 t}\|\leq Mt^{-\a}\e^{-\frac{3}{4}\beta t},\Hs t>0,\label{e3.4}
\end{align}
where $\Lam=A+aI$, and $a$ is a positive number such that ${\rm Re} \,\sig(\Lam)>0$. The first estimates in \eqref{e3.2} and \eqref{e3.3} hold true as the spaces $X^1$ and $X^2$ are finite dimensional.

For a given $\mu\geq 0$, we define a space by
\be\label{aa}
  \cX_\mu=\big\{u\in \cC(\mathbb{R};X^\a):\sup\limits_{t\in \mathbb{R}}\e^{-\mu |t|}\|u(t)\|_\a<\infty\big\},
\ee
which is equipped with the following norm
$$
  \|u\|_{\cX_\mu}=\sup\limits_{t\in \mathbb{R}}\e^{-\mu |t|}\|u(t)\|_\a<\infty.
$$
Then it can be easily seen that $\cX_\mu$ is a Banach space.

Now we rewrite the equation \eqref{e3.1} as
\be\label{e3.5}
u_t+A_\lam u=f(u).
\ee
It is well-known that the Cauchy problem of \eqref{e3.5} is well-posed in $X^\a$; see e.g., \cite[Theorem 3.3.3]{Hen}. Specifically, for each $u_0\in X^\a,$ there exist $T>0$ and a (unique) continuous function $u:[0,T)\ra X^\a$ with $u(0)=u_0$ such that the equation \eqref{e3.5} is satisfied on $(0,T)$.

Denote $\Phi_\lam$ the semiflow generated by the equation \eqref{e3.5}.

Let us start with a fundamental result on the solution of \eqref{e3.5}.

\begin{lem}\label{l3.1}
Let $\mu\in(\frac{1}{4}\beta,\frac{3}{4}\beta)$. Assume $u\in \cX_\mu.$ Then $u$ is a full solution of \eqref{e3.5} if and only if $u$ satisfies
\begin{align}
 u(t)=&\e^{-A^2 t}P_2u(0)+\!\int^t_0\e^{-A^2(t-s)}P_2f(u(s)){\rm d}s+\!\int^t_{-\infty}\e^{-A^3(t-s)}P_3f(u(s)){\rm d}s\nonumber\\
 &-\int^\infty_t\e^{-A^1(t-s)}P_1f(u(s)){\rm d}s.\label{e3.6}
\end{align}
\end{lem}

\begin{proof}
Assume $u\in \cX_\mu$. Let $u$ be a full solution of \eqref{e3.5}. Then
\be\label{e3.7}
  u(t)=u^1(t)+u^2(t)+u^3(t), \hs t\in\R,
\ee
where $u^i(t)=P_iu(t),i\in\{1,2,3\}.$ It is easy to see that $u^i(t)$ can be expressed, respectively, as
\begin{align}
  &u^2(t)=\e^{-A^2 t}u^2(0)+\int^t_{0}\e^{-A^2(t-s)}P_2f(u(s)){\rm d}s,\hs t\in \mathbb{R},\label{e3.8}\\
  &u^1(t)=\e^{-A^1 (t-t_0)}u^1(t_0)+\int^t_{t_0}\e^{-A^1(t-s)}P_1f(u(s)){\rm d}s,\hs t\leq t_0,\label{e3.9}\\
  &u^3(t)=\e^{-A^3 (t-t_0)}u^3(t_0)+\int^t_{t_0}\e^{-A^3(t-s)}P_3f(u(s)){\rm d}s,\hs t\geq t_0.\label{e3.10}
\end{align}
By \eqref{e3.2} (note that $t\leq t_0$), one deduces that
\begin{align*}
  \|\e^{-A^1 (t-t_0)}u^1(t_0)\|_\a
\leq& \,M\e^{\frac{3}{4}\beta(t-t_0)}\|u(t_0)\|_\a\\
=& \,M\e^{\frac{3}{4}\beta t}\e^{-(\frac{3}{4}\beta-\mu)t_0}\(\e^{-\mu t_0}\|u(t_0)\|_\a\)\\
\leq &\,M\e^{\frac{3}{4}\beta t}\e^{-(\frac{3}{4}\beta-\mu)t_0}\|u\|_{\cX_\mu}\ra 0\hs \mb{as}\hs t_0\ra \infty.
\end{align*}
Moreover, we infer from \eqref{e3.4} that
\begin{align*}
  \|\e^{-A^3 (t-t_0)}u^3(t_0)\|_\a
\leq& \,M\e^{-\frac{3}{4}\beta(t-t_0)}\|u(t_0)\|_\a\\
=& M\,\e^{-\frac{3}{4}\beta t}\e^{(\frac{3}{4}\beta-\mu)t_0}(\e^{\mu t_0}\|u(t_0)\|_\a)\\
\leq &\,M\e^{-\frac{3}{4}\beta t}\e^{(\frac{3}{4}\beta-\mu)t_0}\|u\|_{\cX_\mu}\ra 0\hs \mb{as}\hs t_0\ra -\infty.
\end{align*}
Letting $t_0\ra \infty$ and $t_0\ra -\infty$ in \eqref{e3.9}, \eqref{e3.10}, respectively, we conclude from \eqref{e3.7}-\eqref{e3.10} that the integral equation \eqref{e3.6} holds.
\vs
On the other hand, if \eqref{e3.6} holds, then one can easily check that $u(t)$ is a full solution of \eqref{e3.5}.
\end{proof}

\subsection{Existence of global invariant manifolds}
In this subsection, we state and prove the existence result of global invariant manifolds for the equation \eqref{e3.5}. For this purpose, we consider the space $\cX_{\mu}$ for $\mu\in (\frac{1}{4}\beta,\frac{3}{4}\beta)$ being fixed, which is defined by \eqref{aa}.

Henceforth we always assume that the following condition on $f$ holds true.
\vs
($\mathbf{A}$) The Lipschitz constant $L_f$ of $f$ satisfies
$$
  F_\mu L_f<1,
  $$
where
$$
  F_\mu=M\big[\int^\infty_0\e^{-(\mu-\frac{1}{4}\beta)s}\d s+\int^\infty_0(1+s^{-\a})\e^{-(\frac{3}{4}\beta-\mu)s}\d s\big].
$$
Set
$$
  X_i^\a=X^i\cap X^\a,\hs X_{i,j}^\a=X_i^\a\oplus X_j^\a,
$$
where $i,j=1,2,3,i\neq j$. Then we have the following result.

\begin{thm}\label{t3.1}
Let the assumption {\rm(}$\mathbf{A}${\rm)} hold and $\mu\in(\frac{1}{4}\beta,\frac{3}{4}\beta)$ be fixed. Then there exists a Lipschitz continuous mapping $\xi_\lam$ from $X^\a_2$ to $X^\a_{13}$ for $\lam\in J$, such that for each $\lam\in J,$ the system \eqref{e3.5} has a global invariant manifold $\cM_\lam$, which is represented as
$$
  \cM_\lam=\{w+\xi_\lam(w):w\in X^\a_2\}.
$$
\end{thm}

\begin{proof}
Let $\lam\in J$ and $w\in X^\a_2.$ We use the integral equation \eqref{e3.6} to define a mapping $G:=G_{\lam,w}$ on $\cX_{\mu}$ by
\begin{align*}
G(u)(t)=&\e^{-A^2 t}w+\!\int^t_0\e^{-A^2(t-s)}P_2f(u(s)){\rm d}s+\!\int^t_{-\infty}\e^{-A^3(t-s)}P_3f(u(s)){\rm d}s\\
 &-\int^\infty_t\e^{-A^1(t-s)}P_1f(u(s)){\rm d}s.
\end{align*}
Let us first verify that $G: \cX_{\mu}\ra \cX_{\mu}.$ Assume $u\in \cX_\mu.$
Since $f(u)$ is globally Lipschitz continuous, there exists $C_f>0$ such that
$$
  \|f(u)\|\leq C_f(\|u\|_\a+1).
$$
By \eqref{e3.2}-\eqref{e3.4}, we have
\begin{align}\label{e3.11}
\|G(u)\|_\a\leq& M\e^{\frac{1}{4}\beta|t|}\|w\|_\a+M\int^{t_2}_{t_1}\e^{\frac{1}{4}\beta|t-s|}C_f(\|u(s)\|_\a+1){\rm d}s\nonumber\\
&+M\int_{-\infty}^t(t-s)^{-\a} \e^{-\frac{3}{4}\beta(t-s)}C_f(\|u(s)\|_\a+1){\rm d}s\nonumber\\
&+M\int^\infty_t\e^{\frac{3}{4}\beta(t-s)}C_f(\|u(s)\|_\a+1){\rm d}s,
\end{align}
where
$$
  t_1:=\min\{t,0\},\hs t_2:=\max\{t,0\}.
  $$
In what follows we consider the second term on the right-hand side of the above inequality.
If $t\geq 0,$ then
\begin{align}\label{e3.12}
 &\e^{-\mu|t|}M\int^{t_2}_{t_1}\e^{\frac{1}{4}\beta|t-s|}C_f(\|u(s)\|_\a+1){\rm d}s\nonumber \\
 =&\e^{-\mu t}M\int^{t}_{0}\e^{\frac{1}{4}\beta(t-s)}C_f(\|u(s)\|_\a+1){\rm d}s\nonumber\\
=&MC_f\int^{t}_{0}\e^{-(\mu-\frac{1}{4}\beta)(t-s)}\e^{-\mu s}(\|u(s)\|_\a+1){\rm d}s\nonumber\\
\leq& MC_f\int^{t}_{0}\e^{-(\mu-\frac{1}{4}\beta)(t-s)}\(\e^{-\mu s}\|u(s)\|_\a+1\){\rm d}s.
\end{align}
Similarly, for $t\leq 0,$ we have
\begin{align}\label{e3.13}
 & \e^{-\mu|t|}M\int^{t_2}_{t_1}\e^{\frac{1}{4}\beta|t-s|}C_f(\|u(s)\|_\a+1){\rm d}s\nonumber \\
 =&\e^{\mu t}M\int^{0}_{t}\e^{\frac{1}{4}\beta(s-t)}C_f(\|u(s)\|_\a+1){\rm d}s\nonumber\\
=&MC_f\int^{0}_{t}\e^{-(\mu-\frac{1}{4}\beta)(s-t)}\e^{\mu s}(\|u(s)\|_\a+1){\rm d}s\nonumber\\
\leq& MC_f\int^{0}_{t}\e^{-(\mu-\frac{1}{4}\beta)(s-t)}\(\e^{\mu s}\|u(s)\|_\a+1\){\rm d}s.
\end{align}
Thus it follows from \eqref{e3.12} and \eqref{e3.13} that for $t\in \R$,
\begin{align}\label{e3.14}
 &\e^{-\mu|t|}M\int^{t_2}_{t_1}\e^{\frac{1}{4}\beta|t-s|}C_f(\|u(s)\|_\a+1){\rm d}s\nonumber\\
\leq& M C_f\int^{t_2}_{t_1}\e^{-(\mu-\frac{1}{4}\beta)|t-s|}\big(\e^{-\mu|s|}\|u(s)\|_\a+1\big){\rm d}s\nonumber \\
\leq &M C_f\int^{\8}_{0}\e^{-(\mu-\frac{1}{4}\beta)s}{\rm d}s\cdot(\|u\|_{\cX_\mu}+1).
\end{align}
We observe that
\begin{align*}
 \e^{-\mu|t|}=\e^{-\mu|t-s+s|}\leq \e^{-\mu|s|}\e^{\mu|t-s|},\hs t\in \R,
\end{align*}
from which it can be seen that 
\begin{align}\label{e3.15}
&\e^{-\mu|t|}M\int_{-\infty}^t(t-s)^{-\a} \e^{-\frac{3}{4}\beta(t-s)}C_f(\|u(s)\|_\a+1)\d s\nonumber \\
\leq&
M C_f\int_{-\infty}^t(t-s)^{-\a} \e^{\mu|t-s|}\e^{-\frac{3}{4}\beta(t-s)}\big[\e^{-\mu|s|}(\|u(s)\|_\a+1)\big]\d s\nonumber \\
\leq&M C_f\int_{-\infty}^t(t-s)^{-\a} \e^{-(\frac{3}{4}\beta-\mu)(t-s)}\(\e^{-\mu|s|}\|u(s)\|_\a+1\)\d s\nonumber \\
\leq &M C_f\int^\infty_0s^{-\a}\e^{-(\frac{3}{4}\beta-\mu)s}\d s\cdot(\|u\|_{\cX_\mu}+1),
\end{align}
and
\begin{align}\label{e3.16}
&\e^{-\mu|t|}M\int^\infty_t\e^{\frac{3}{4}\beta(t-s)}C_f(\|u(s)\|_\a+1)\d s \nonumber \\
\leq &M C_f\int^\infty_t\e^{\mu|t-s|}\e^{\frac{3}{4}\beta(t-s)}\e^{-\mu|s|}\(\|u(s)\|_\a+1\)\d s \nonumber \\
\leq &M C_f\int^\infty_t\e^{(\frac{3}{4}\beta-\mu)(t-s)}\(\e^{-\mu|s|}\|u(s)\|_\a+1\)\d s\nonumber \\
\leq &M C_f\int^\infty_0\e^{-(\frac{3}{4}\beta-\mu)s}{\d} s\cdot(\|u\|_{\cX_\mu}+1).
\end{align}
Thus one can conclude from \eqref{e3.11}, \eqref{e3.14}, \eqref{e3.15} and \eqref{e3.16} that
\begin{align*}
\e^{-\mu|t|}\|G(u)\|_\a
\leq& M\e^{-(\mu-\frac{1}{4}\beta)|t|}\|w\|_\a+\e^{-\mu|t|}M\int^{t_2}_{t_1}\e^{\frac{1}{4}\beta|t-s|}C_f(\|u(s)\|_\a+1){\rm d}s\nonumber\\
&+\e^{-\mu|t|}M\int_{-\infty}^t(t-s)^{-\a} \e^{-\frac{3}{4}\beta(t-s)}C_f(\|u(s)\|_\a+1){\rm d}s\nonumber\\
&+\e^{-\mu|t|}M\int^\infty_t\e^{\frac{3}{4}\beta(t-s)}C_f(\|u(s)\|_\a+1){\rm d}s\nonumber \\
\leq&M\|w\|_\a+M C_f\int^{\8}_{0}\e^{-(\mu-\frac{1}{4}\beta)s}{\rm d}s\cdot(\|u\|_{\cX_\mu}+1)\nonumber\\
&+M C_f\int^\infty_0s^{-\a}\e^{-(\frac{3}{4}\beta-\mu)s}\d s\cdot(\|u\|_{\cX_\mu}+1)\nonumber \\
&+M C_f\int^\infty_0\e^{-(\frac{3}{4}\beta-\mu)s}{\d} s\cdot(\|u\|_{\cX_\mu}+1)\nonumber \\
\leq& M\|w\|_\a+F_\mu C_f(\|u\|_{\cX_\mu}+1)<\infty,\hs \forall t\in \mathbb{R}.
\end{align*}
%
That is $\|Gu\|_{\cX_\mu}<\infty,$ which implies $Gu\in \cX_\mu.$
\vs
Next we show that $G$ is a contraction mapping on $\cX_\mu.$ Indeed, for each $u_1,u_2\in \cX_\mu,$ we have
\begin{align}\label{e3.17}
\e^{-\mu|t|}\|G(u_1)-G(u_2)\|_\a
&\leq \e^{-\mu|t|}\|\int^{t}_{0}\e^{-A^2(t-s)}P_2\big(f(u_1)-f(u_2)\big)\d s\|_\a\nonumber\\
&\, +\e^{-\mu|t|}\|\int^t_{-\infty}\!\e^{-A^3(t-s)}P_3\big(f(u_1)-f(u_2)\big){\rm d}s\|_\a\nonumber \\
&\, +\e^{-\mu|t|}\|\int^\infty_t\!\e^{-A^1(t-s)}P_1\big(f(u_1)-f(u_2)\big)\d s\|_\a.
\end{align}
Similar to the derivation of \eqref{e3.14}, it can be easily shown that
\begin{align}\label{e3.18}
&\e^{-\mu|t|}\|\int^{t}_{0}\e^{-A^2(t-s)}P_2\big(f(u_1)-f(u_2)\big)\d s\|_\a\nonumber\\
\leq& \,ML_f\int^{t_2}_{t_1}\e^{-(\mu-\frac{1}{4}\beta)|t-s|}\big(\e^{-\mu|s|}\|u_1-u_2\|_\a\big)\d s.
\end{align}
By some computations, we obtain that
\begin{align}
&\e^{-\mu|t|}\|\int^t_{-\infty}\e^{-A^3(t-s)}P_3\big(f(u_1)-f(u_2)\big){\rm d}s\|_\a\nonumber \\
\leq &\,ML_f\int^t_{-\infty}(t-s)^{-\a}\e^{-(\frac{3}{4}\beta-\mu)(t-s)}\big(\e^{-\mu|s|}\|u_1-u_2\|_\a\big)\d s,\label{e3.19}\\
&\e^{-\mu|t|}\|\int^\infty_t\e^{-A^1(t-s)}P_1(f(u_1)-f(u_2))\d s\|_\a\nonumber \\
\leq& \, ML_f\int_t^{\infty}\e^{(\frac{3}{4}\beta-\mu)(t-s)}\big(\e^{-\mu|s|}\|u_1-u_2\|_\a\big)\d s.\label{e3.20}
\end{align}
Then it follows from \eqref{e3.17}-\eqref{e3.20} that
\begin{align}\label{e3.21}
\e^{-\mu|t|}\|Gu_1-Gu_2\|_{\a}&\leq ML_f\int^{t_2}_{t_1}\e^{-(\mu-\frac{1}{4}\beta)|t-s|}\big(\e^{-\mu|s|}\|u_1-u_2\|_\a\big)\d s\nonumber\\
&+ML_f\int^t_{-\infty}(t-s)^{-\a}\e^{-(\frac{3}{4}\beta-\mu)(t-s)}\!\big(\e^{-\mu|s|}\|u_1-u_2\|_\a\big)\d s\nonumber\\
&+ML_f\int_t^{\infty}\e^{(\frac{3}{4}\beta-\mu)(t-s)}\big(\e^{-\mu|s|}\|u_1-u_2\|_\a\big)\d s\nonumber\\
&\leq F_\mu L_f\|u_1-u_2\|_{\cX_\mu}.
\end{align}
Therefore
$$\|Gu_1-Gu_2\|_{\cX_\mu}\leq F_\mu L_f\|u_1-u_2\|_{\cX_\mu}.$$
That is, $G$ is a contraction mapping on $\cX_\mu$.

Now by virtue of the Banach contraction mapping principle, $G$ has a fixed point $u_w:=u_{\lam,w}\in \cX_\mu,$ which is a full solution of \eqref{e3.5} with $P_2u_w(0)=w$ and satisfies
\begin{align}\label{e3.22}
   u_w(t)=&\e^{-A^2 t}w+\!\int^t_0\e^{-A^2(t-s)}P_2f(u_w(s))\d s+\!\int^t_{-\infty}\e^{-A^3(t-s)}P_3f(u_w(s))\d s\nonumber\\
 &-\int^\infty_t\e^{-A^1(t-s)}P_1f(u_w(s))\d s.
\end{align}
Define
$$
\Gam(w)=u_w(0),\hs w\in X^\a_2.
$$
We claim that $\Gam(w)$ is Lipschitz continuous with respect to $w$. Indeed, let $w_1,w_2\in X_2^\a$. Then by \eqref{e3.22} and using the same estimation of \eqref{e3.17}, we obtain that
\begin{align*}
  &\e^{-\mu|t|}\|u_{w_1}(t)-u_{w_2}(t)\|_\a\\
  \leq&\, M\e^{-(\mu-\frac{1}{4}\beta)|t|}\|w_1-w_2\|_\a +ML_f\int^{t_2}_{t_1}\e^{-(\mu-\frac{1}{4}\beta)|t-s|}\big(\e^{-\mu|s|}\|u_{w_1}-u_{w_2}\|_\a\big) \d s\\
&+ML_f\int^t_{-\infty}(t-s)^{-\a}\e^{-(\frac{3}{4}\beta-\mu)(t-s)}\big(\e^{-\mu|s|}\|u_{w_1}-u_{w_2}\|_\a\big)\d s\\
&+ML_f\int_t^{\infty}\e^{(\frac{3}{4}\beta-\mu)(t-s)}\big(\e^{-\mu|s|}\|u_{w_1}-u_{w_2}\|_\a\big)\d s\\
\leq&\, M\|w_1-w_2\|_\a+F_\mu L_f\|u_{w_1}-u_{w_2}\|_{\cX_\mu},\hs t\in \R,
\end{align*}
from which it can be seen that
$$\|u_{w_1}-u_{w_2}\|_{\cX_\mu}\leq M\|w_1-w_2\|_\a+F_\mu L_f\|u_{w_1}-u_{w_2}\|_{\cX_\mu}.$$
Thus
\begin{align*}
\|\Gam(w_1)-\Gam(w_2)\|_\a&=\|u_{w_1}(0)-u_{w_2}(0)\|_\a\leq \|u_{w_1}(t)-u_{w_2}(t)\|_{\cX_\mu}\\[1ex]
&\leq M(1-F_\mu L_f)^{-1}\|w_1-w_2\|_\a,
\end{align*}
which shows that $\Gam(w)$ is Lipschitz continuous in $w$.

Define a mapping $\xi_\lam:X_2^\a\ra X^\a_{13}$ by
\begin{align}\label{e3.23}
\xi_\lam(w)=\int^0_{-\infty}\e^{A^3s}P_3f(u_w(s))\d s
 -\int^\infty_0\e^{A^1s}P_1f(u_w(s))\d s, \hs w\in X^\a_2.
\end{align}
Let $t=0$ in \eqref{e3.22}. Then we have
\be\label{e3.24}
\Gam(w)=u_w(0)=w+\xi_\lam(w),\hs w\in X^\a_2.
\ee
Clearly, one can deduce from \eqref{e3.24} that $\xi_\lam:X^\a_2\ra X^\a_{13}$ is Lipschitz continuous in $w.$

Set
$$
  \cM_\lam=\{w+\xi_\lam(w):w\in X^\a_2\}.
  $$
Then $\cM_\lam$ is a global invariant manifold for the semiflow generated by \eqref{e3.5}.
\end{proof}

\begin{remark}\label{r3.1}
By \eqref{e3.23}, one can easily deduce that if $f$ is bounded, then there exists $M_f>0$ depending on $f$ such that
$$
  \|\xi_\lam(w)\|_\a\leq MM_f\int_0^{\infty}(1+s^{-\a})\e^{-\frac{3}{4}\beta s}\d s, \hs w\in X^\a_2,
$$
that is, $\xi_\lam$ is bounded.
\end{remark}

\setcounter {equation}{0}
\section{Bifurcation from infinity of Schr\"odinger equations}

In this section we study the bifurcation from infinity of the Schr\"odinger equation \eqref{e1.1}.
For this purpose, let us first consider the following parabolic equation
\be\label{e4.2}\left\{\ba{ll}
 u_t-\Delta u+V(x)u=\lam u+f(x,u),\hs x\in \mathbb{R}^N;\\[1ex]
u(0,x)=u_0(x),\hs x\in \mathbb{R}^N.
\ea\right.\ee
where $\lam\in \mathbb{R}$, $V\in L^\infty(\mathbb{R}^N), N\geq 1$, and $f$ is bounded and satisfies the Landesman-Lazer type condition \eqref{1.2}.
We will apply the global invariant manifold established in Section 3 to discuss the dynamic bifurcation from infinity of \eqref{e4.2}.

\subsection{Mathematical setting}



Let $X=L^2(\mathbb{R}^N)$ and $Y=H^1(\mathbb{R}^N)$. Denote $(\cdot,\cdot)$
and $|\cdot|$ the usual inner product and norm on $X$,
respectively. The norm $\|\cdot\|$ on $Y$ is defined by
$$
  \|u\|=\big(\int_{\mathbb{R}^N}|\nabla u|^2\mathrm{d}x+\int_{\mathbb{R}^N}V(x)|u|^2\mathrm{d}x\big)^{1/2},\Hs
  u\in Y.
$$
For simplicity, from now on we use $\mathrm{B}_X(R)$ and $\mathrm{B}_Y(R)$ to denote
the balls in $X$ and $Y$, centered at $0$ with radius $R$, respectively.


\vs
In the following we assume that the conditions on the potential $V$ and nonlinearity $f$ hold true.
 \benu
\item[($\mathbf{A1}$)] There exist positive numbers $a_1,a_2$ and $V_\infty$ such that
$$
  0<a_1\leq V(x)\leq a_2,\hs 0<V_\infty:=\sup_{x\in\R^N}V(x)=\lim\limits_{|x|\ra \infty}V(x)<\infty.
$$
\item[($\mathbf{A2}$)] The nonlinear term $f:\R^N\X \R \ra \R$ satisfies
$$
 |f(x,s)-f(x,t)|\leq l(x)|s-t|,\Hs \forall s,t\in \R, \hs x\in \R^N,
$$
where $l(x)=l_1(x)+l_2(x)$, $l_1\in L^\8(\R^N)$, and $l_2$ satisfies
$$l_2\in L^p(\R^N),\hs\mb{$p\geq 2$ if $N=1$,\, $p>2$ if $N=2$, \,$p\geq N$ if $N\geq 3$.}$$
\item[($\mathbf{A3}$)] The function $f$ satisfies the Landesman-Lazer type condition \eqref{1.3} in Section 1, and
$$
  |f(x,s)|\leq g(x), \Hs \A s\in\R,\hs x\in \R^N
$$
for some function $g\in L^2(\R^N)\cap L^\8(\R^N)$.
\eenu

Denote by $A$ the operator $-\Delta+V(x)$.
Then under the assumptions on $V(x),$ the Schr\"odinger operator $A=-\Delta+V(x):H^2(\mathbb{R}^N)\ra L^2(\mathbb{R}^N)$ is selfadjoint
and bounded from below.
Furthermore, it follows from \cite{BS,RS,WX} that the interval $[V_\infty,\infty)$ is the essential spectrum of $A$ and the discrete spectrum of $A$ on the interval $(-\8,V_\8)$ appears. Namely, for each $a<V_\8$, $\sig(A)\cap(-\8,a)$ consists of at most finitely many eigenvalues of $A$. Hence 
$$
  \sig(A)=\sig_p\cup\sig_e,
  $$
where $\sig_p$ consists of isolated eigenvalues with finite multiplicity, and $\sig_e=[V_\infty,\infty)$ is the essential spectrum.

Now we convert \eqref{e4.2} into an abstract form on $Y$
\begin{align}\label{4.3}
 u_t+Au=&\lam u+\tilde{f}(u), \\
 u(0)=&u_0,\label{4.4}
\end{align}
where $\tilde{f}(u)$ is the Nemitski operator from $Y$ to $X$ defined by
$$
  \tilde{f}(u)(x)=f(x,u),\Hs u\in Y.
$$
It is well known that the Nemitski operator  $\tilde{f}$ is well defined. Moreover, by ($\mathbf{A2}$), one can easily verify that $\tilde{f}$ is Lipschitz continuous with a chosen Lipschitz constant denoted by $L_f$.
Note that $A$ is a sectorial operator. Then the Cauchy problem of \eqref{4.3}-\eqref{4.4} is well-posed; see \cite[Theorem 3.3.3]{Hen}.
Denote $\Phi_{\lam}$ the semiflow generated by \eqref{4.3}-\eqref{4.4}. Then $\Phi_\lam$ is a global semiflow on $Y$.

Let $A_{\lam_*}=A-\lam_*$, where $\lam_*\in \sig_p$. Then the space $X$ can be
decomposed into the orthogonal direct sum of its subspaces
$X^1,X^2,X^3$ corresponding to the negative, zero and positive
eigenvalues of $A_{\lam_*}$, respectively. Denote $P_\sig (\sig\in \{1,2,3\})$ the projection from $X$ to $X^\sig.$ It can be easily seen that $X^1$
and $X^2$ are finite-dimensional.

Set
$$
  Y^\sigma=Y\cap X^\sigma,\Hs \sigma\in\{1,2,3\}.
  $$
Then from the finite dimensionality of $X^1$ and $X^2$, one
concludes that $Y^1$ and $Y^2$ coincide with $X^1$ and $X^2$,
respectively.
Also, it holds that
$$
  Y=Y^1\oplus Y^2\oplus Y^3,
  $$
and $m:=\dim Y^2\geq 1$.
Since $\lam_*$ is an isolated eigenvalue of $A$, one can take the same positive number $\beta$ as that in Section 3.1.
Let $J=(\lam_*-\frac{1}{4}\beta,\lam_*+\frac{1}{4}\beta).$

Henceforth we always assume that the Lipschitz constant $L_f$ of $f$ satisfies the following condition.
\vs
({\bf F}) For each fixed $\mu\in (\frac{1}{4}\beta,\frac{3}{4}\beta)$,
$$
  F_{\mu}L_f<1,
$$
where
\be\label{L}
  F_{\mu}=M\big(\int^\infty_0\e^{-(\mu-\frac{1}{4}\beta)s}\d s
  +\int^\infty_0(1+s^{-\frac{1}{2}})\e^{-(\frac{3}{4}\beta-\mu)s}\d s\big),
\ee
and $M$ is a positive number depending on $A$.
Then one can immediately conclude from Theorem 3.1 that the following result holds true.

\begin{thm}\label{t4.1}
Let the assumptions {\rm($\mathbf{A1}$)}-{\rm($\mathbf{A3}$)} and {\rm ($\mathbf{F}$)} hold. Then there exists a Lipschitz continuous mapping $\xi_\lam$ from $Y^2$ to $Y^{13}$ for $\lam\in J$, such that for each $\lam\in J,$ the system \eqref{4.3}-\eqref{4.4} has a global invariant manifold $\cM_\lam$, which is represented as
$$
  \cM_\lam=\{w+\xi_\lam(w):w\in Y^2\},
$$
where $Y^{13}=Y^1\oplus Y^3$.
\end{thm}
\subsection{Preliminaries}

Given a function $w$ on $\R^N$, we use $w_{\pm}$ to denote the
positive and negative parts of $w$, respectively. Specifically,
$$
  w_{\pm}=\max\{\pm w(x),0\}, \Hs x\in \R^N.
  $$
Then $w=w_+-w_-.$
 \vs

\begin{lem}\label{l4.1} Let the assumption {\rm($\mathbf{A3}$)} hold. Then
for any $R>0,\varepsilon>0,$ there exists $s_0>0$ such that if $s\geq s_0$,
$$
  \int_{\mathbb{R}^N}f(x,v+sw)w{\rm d}x\geq
  \frac{1}{2}\int_{\mathbb{R}^N}(\bar{f}w_++\underline{f}w_-){\rm d}x-\varepsilon.
  $$
for all  $v\in {\rm\bar{B}}_X(R)$ and $w\in {\rm\bar{B}}_{X^2}(1),$ where $X^2$ is a finite-dimensional Banach space.
\end{lem}

\begin{proof}
Fix a positive integer $n$ sufficiently large. Denote $\Omega_n$ a ball in $\R^N$ centered at $0$ with radius $n$.
Let
$$
  I=\int_{\Omega_n}f(x,v+sw)w{\rm d}x-\frac{1}{2}\int_{\Omega_n}(\bar{f}w_++\underline{f}w_-){\rm d}x.
  $$
Noticing that $w=w_+-w_-,$ we can rewrite $I$ as $I=I_+-I_-,$ where
$$
  I_+=\int_{\Omega_n}(f(x,v+sw)-\bar{f}/2)w_+{\rm d}x,\hs I_-=\int_{\Omega_n}(f(x,v+sw)+\underline{f}/2)w_-{\rm d}x.$$
In what follows we first estimate $I_+$ for  $v\in {\rm\bar{B}}_X(R)$ and $w\in {\rm\bar{B}}_{X^2}(1)$.

Write $\{|v|\geq \sig\}=\{x\in\Omega_n:|v|\geq \sig\}$. Since
$$
  R^2\geq \int_{\mathbb{R}^N}|v|^2{\rm d}x \geq \int_{\Omega_n}|v|^2{\rm d}x\geq \int_{\{|v|\geq \sig\}}|v|^2{\rm d}x\geq \sig^2|\{|v|\geq \sig\}|,$$
one can deduce that $|\{|v|\geq \sig\}|\ra 0$ as $\sig \ra \infty$ uniformly with respect to $v\in {\rm\bar{B}}_X(R).$
Thus we can pick a $\sig>0$ sufficiently large so that
$$
  |\{|v|\geq \sig\}|^{1/2}<\de:=\ve/\big(8f_\infty(|\Omega_n|+1)\big),$$
where $f_\8=\sup_{x\in \R^N,s\in\R}|f(x,s)|$.

For each $v\in {\rm\bar{B}}_X(R)$ and $w\in {\rm\bar{B}}_{X^2}(1),$ let $D_n=\{|v|< \sig\}\cap\{w_+> \de\}.$ Then
$$
  \Omega_n=D_n\cup\{|v|\geq \sig\}\cup\{w_+\leq \de\}.$$
Therefore
\begin{align*}
  I_+&\geq\int_{D_n}(f(x,v+sw)-\bar{f}/2)w_+{\rm d}x-\int_{\{|v|\geq \sig\}}|f(x,v+sw)-\bar{f}/2|w_+{\rm d}x\\
  &\hs -\int_{\{w_+\leq \de\}}|f(x,v+sw)-\bar{f}/2|w_+{\rm d}x\\
  &\geq\int_{D_n}(f(x,v+sw)-\bar{f}/2)w_+{\rm d}x-2f_\infty\big(\int_{\{|v|\geq \sig\}}w_+\d x+\int_{\{w_+\leq \de\}}w_+{\rm d}x\big).
\end{align*}
It is easy to see that
\begin{align*}
  \int_{\{|v|\geq \sig\}}w_+{\rm d}x\leq \big(\int_{\{|v|\geq \sig\}}w_+^2{\rm d}x\big)^{1/2}|\{|v|\geq \sig\}|^{1/2}
  \leq |w|\de\leq \de.
\end{align*}
Note that
$$
  \int_{\{w_+\leq \de\}}w_+{\rm d}x\leq |\Omega_n|\de.
$$
Thereby
\begin{align}\label{4.5}
  I_+&\geq\int_{D_n}(f(x,v+sw)-\bar{f}/2)w_+{\rm d}x-2f_\infty(|\Omega_n|+1)\de\nonumber\\
  &=\int_{D_n}(f(x,v+sw)-\bar{f}/2)w_+{\rm d}x-\frac{\ve}{4}.
\end{align}
Since $z+s\eta \ra +\infty$ (as $s\ra +\infty$) uniformly with respect to $z\in[-\sig,\sig]$ and $\eta\geq \de,$ there exists $s_1>0$ (independent of $n$) such that if $s>s_1,$
$$
  f(x,z+s\eta)-\bar{f}/2\geq -\frac{\ve}{4|\Omega_n|^{1/2}},\Hs z\in [-\sig,\sig],\hs \eta\geq \de.
$$
Let $s>s_1.$ Then by the definition of $D_n$ (note that $w=w_+$ on $D_n$), we deduce that
\begin{align*}
  \int_{D_n}(f(x,v+sw)-\bar{f}/2)w_+{\rm d}x&\geq -\frac{\ve}{4|\Omega_n|^{1/2}}\int_{D_n}w_+{\rm d}x\\
  &\geq -\frac{\ve}{4|\Omega_n|^{1/2}}|D_n|^{1/2}(\int_{D_n}|w|^2{\rm d}x)^{1/2}\geq -\frac{\ve}{4}.
\end{align*}
It follows from \eqref{4.5} that if $s>s_1$,
\be\label{e:4.6}
  I_+\geq \int_{D_n}(f(x,v+sw)-\bar{f}/2)w_+{\rm d}x-\frac{\ve}{4}\geq -\frac{\ve}{2},
\ee
uniformly with respect to the positive integer $n$.
\vs
Define $g_n(x)=\chi_n(x)[(f(x,v+sw)-\bar{f}/2)w_+],$ where
$$
\chi_n(x)=\left\{\ba{ll}1, \hs &x\in \Omega_n; \\[1ex]
0,& x\in \R^N \backslash \Omega_n.\ea\right. $$
Then
$$
  \lim_{n\ra \infty}g_n(x)=\big(f(x,v+sw)-\bar{f}/2\big)w_+.$$
Noticing that
$$
  I_+=\int_{\Omega_n}(f(x,v+sw)-\bar{f}/2)w_+{\rm d}x=\int_{\R^N}\chi_n(x)[(f(x,v+sw)-\bar{f}/2)w_+]{\rm d}x,$$
by \eqref{e:4.6} we have
\be\label{e:4.7}
  \int_{\R^N}\chi_n(x)[(f(x,v+sw)-\bar{f}/2)w_+]{\rm d}x\geq -\frac{\ve}{2},
\ee
provided $s>s_1$, uniformly with respect to $n$. Since the norm $\|\cdot\|_{L^1(\R^N)}$ and that of $X=L^2(\R^N)$ are equivalent on $X^2$, one easily sees that $$\sup\{\|w\|_{L^1(\R^N)}: w\in X^2, |w|=1\}<\infty.$$
Therefore
\begin{align*}
  \int_{\R^N}|\chi_{n}(x)[(f(x,v+sw)-\bar{f}/2)w_+]|{\rm d}x&\leq 2f_\infty\int_{\R^N}|w|{\rm d}x\\
  &\leq 2f_\infty\|w\|_{L^1(\R^N)}<\infty.
\end{align*}
Thanks to the Lebesgue dominated convergence theorem and letting $n\ra \infty$ in \eqref{e:4.7}, we conclude that if $s>s_1$,
$$
  \int_{\R^N}(f(x,v+sw)-\bar{f}/2)w_+{\rm d}x\geq -\frac{\ve}{2}.
$$
Similarly, it can be shown that there exists $s_2>0$ (independent of $n$) such that
$$\int_{\R^N}(f(x,v+sw)+\bar{f}/2)w_-\d x\leq \frac{\ve}{2},$$ as long as $s>s_2.$
\vs
Set $s_0=\max\{s_1,s_2\}$. Then if $s>s_0,$ we have
$$
  \int_{\R^N}(f(x,v+sw)-\bar{f}/2)w_+{\rm d}x-\!\int_{\R^N}(f(x,v+sw)+\bar{f}/2)w_-{\rm d}x\geq \!-\frac{\ve}{2}-\frac{\ve}{2}=-\ve$$
for all $v\in {\rm\bar{B}}_X(R)$ and $w\in {\rm\bar{B}}_{X^2}(1)$. The proof of the Lemma is complete.
\end{proof}
\br\label{l4.1} In \cite{LLZ}, the authors used the Landesman-Lazer type condition to give a fundamental result on $f$ (see \cite[Lemma 5.2]{LLZ}), which plays an important role in establishing their main results on bifurcations from infinity and multiplicity of solutions of parabolic equations on bounded domains. Lemma \ref{l4.1} significantly extends this result to the case of unbounded domains.
\er
Now we restrict the system \eqref{4.3}-\eqref{4.4} to the global invariant manifold $\cM_\lam$ given by Theorem \ref{t4.1}. Then
\be\label{4.8}
w_t+A^2 w=P_2\tilde{f}(w+\xi_\lam(w)),\hs w\in Y^2,
\ee
where $w=P_2u$, $A^2=A_\lam|_{X^2}, \lam\in J$. Denote $\phi_\lam$ the semiflow generated by \eqref{4.8}. For each $0<a\leq b\leq \infty,$ write
$$
  \mathrm{B}_2[a,b]=\{w\in Y^2, a\leq |w|\leq b\}.
$$

\begin{lem}\label{l4.2}
Let the assumptions {\rm ($\mathbf{A1}$)-($\mathbf{A3}$)} and {\rm ({\bf F})} hold true. Then there exist $R_0>0,c_0>0$ such that the following assertions hold.
\benu
\item[(1)] For $\lam\in[\lam_*,\lam_*+\frac{1}{4}\beta)$, if $w(t)$ is a solution of \eqref{4.8} in $\mathrm{B}_{2}[R_0,\8]$, we have
\be\label{4.9}
\frac{{\rm d}}{{\rm d}t}|w(t)|^2\geq c_0 |w(t)|.
\ee
\item[(2)] If $R>R_0$, then there exists
 $0<\eta\leq\frac{1}{4}\beta$ such that for each $\lam\in[\lam_*-\eta,\lam_*)$, \eqref{4.9} holds true for any solution $w(t)$  of \eqref{4.8} in $\mathrm{B}_2[R_0,R]$.
 \item[(3)] There exists $\theta_1>0$ such that if $\lam\in[\lam_*-\theta_1,\lam_*)$, then the equation \eqref{4.8} has a
 positively invariant set $\mathrm{B}_2[r_\lam,R_\lam]$ with $r_\lam,R_\lam\ra \infty,$ as $\lam\ra \lam_*^{-}$.
 \eenu
\end{lem}

\begin{proof}
Taking the inner product of \eqref{4.8} with $w$ in $X$, we have
\be\label{4.10}
\frac{1}{2}\frac{{\rm d}}{{\rm d}t}|w|^2+\|w\|^2=\lam|w|^2+(\tilde{f}(w+\xi_\lam(w)),w).
\ee
Note that $\|w\|^2=\lam_* |w|^2$. Then it follows from \eqref{4.10} that
\be\label{4.11}
\frac{1}{2}\frac{{\rm d}}{{\rm d}t}|w|^2=(\lam-\lam_*)|w|^2+(\tilde{f}(w+\xi_\lam(w)),w).
\ee
We first estimate the last term in \eqref{4.11}.
Since the norm $\|\cdot\|_{L^1(\mathbb{R}^N)}$ of $L^1(\R^N)$ and that of $X=L^2(\mathbb{R}^N)$ are equivalent in $Y^2,$ one easily sees that
\be\label{4.12}
\min\{\|v\|_{L^1(\mathbb{R}^N)}:\,\,v\in Y^2,\,\,|v|=1\}:=m>0.
\ee
Choose a positive number $\de$ with $\de<\frac{1}{2}\min\{\bar{f},\underline{f}\}$. By virtue of Remark \ref{r3.1}, we see that $\|\xi_\lam(w)\|\leq R_1$ for some $R_1>0$. Thanks to Lemma \ref{l4.1}, there exists $s_0>0$ such that if
$s>s_0$,
\be\label{4.13}
(\~f(h+sv),v)=\int_{\mathbb{R}^N} f(x,h+sv)v\,{\rm d}x\geq
\frac{1}{2}\int_{\mathbb{R}^N}\(\ol fv_++\ul fv_-\){\rm d}x-\frac{1}{2}m\de \ee for all $h\in
\ol\mB_X(R_1)$ and $v\in \ol\mB_{X^2}(1)$.
Now we rewrite
$$w=sv,\hs\mb{where $s=|w|$}.$$
Clearly $v\in \partial \mathrm{B}_{X^2}(1)$. Hence if $s\geq s_0$, we deduce from \eqref{4.13} that
$$
(\tilde{f}(w+\xi_\lam(w)),w)=s(f(x,sv+\xi_\lam(w)),v)\geq s\(\frac{1}{2}\int_{\mathbb{R}^N}\(\ol f v_++\ul
f v_-\){\rm d}x-\frac{1}{2}m\de \).
$$
Note that
\begin{align*}
\frac{1}{2}\int_{\mathbb{R}^N}\(\ol fv_++\ul fv_-\)\d x -\frac{1}{2}m\de
&\geq\de \int_{\mathbb{R}^N}|v|{\rm d}x-\frac{1}{2}m\de\\[2ex]&\geq(\mb{by }\eqref{4.12})\geq \frac{1}{2}m\de.
\end{align*}
Thereby
\be\label{4.14} (\tilde{f}(w+\xi_\lam(w)),w)\geq
\frac{1}{2}m\de s=\frac{1}{2}m\de |w|. \ee
It follows from \eqref{4.11} and \eqref{4.14} that
\be\label{4.15}
\frac{\d}{\d t}|w(t)|^2\geq 2\(\lam-\lam_*\) |w|^2+m\de |w(t)|
\ee as long as $|w(t)|\geq s_0$.
\vs
Now we pick $R_0=s_0$,  $c_0=m\de/2$. Then if $\lam\in[\lam_*,\lam_*+\frac{1}{4}\beta)$, we infer from \eqref{4.15} that
$$
\frac{\d}{\d t}|w(t)|^2\geq m\de |w(t)|>c_0|w(t)|
$$ at any point $t$ where $|w(t)|\geq R_0$. Therefore the assertion  (1) holds.

Assume $R>R_0$ and let $\lam<\lam_*$. Take an $\eta>0$ such that $\eta R^2<m\de s_0/4$. Then for each $\lam\in[\lam_*-\eta,\lam_*)$, if $w(t)$ is a solution of \eqref{4.8} in $\mathrm{B}_2[R_0,R]$, we conclude from \eqref{4.15} that
\begin{align*}
\frac{\d}{\d t}|w(t)|^2&\geq -2|\lam-\lam_*|\,R^2+m\de |w(t)|\\
&\geq c_0|w(t)|+\(c_0|w(t)|-2\eta R^2\)\\
&\geq c_0|w(t)|+ \(c_0s_0-2\eta R^2\) \geq c_0|w(t)|,
\end{align*}  which justifies the second assertion (2).

Finally, we verify that the assertion (3) holds. Take a sequence $R_m>R_0 \,(m=1,2,\cdots)$ such that $R_0<R_1<R_2<\cdots<R_m<\cdots$ and $R_m\ra \8$ as $m\ra \8$. Then for each fixed $m$, we conclude from assertion (2) that there exists a sequence $\eta_m>0 \,(m=1,2,\cdots)$ with
$$\eta_1>\eta_{2}>\cdots>\eta_m>\cdots,\hs \eta_m\ra 0\,(m\ra \8),$$
such that if $\lam\in [\lam_*-\eta_m,\lam_*)$, \eqref{4.9} holds for any solution $w(t)$ of \eqref{4.8} in $\mathrm{B}_2[R_0,R_m]$. Furthermore, by the choice of $\eta_m$ and \eqref{4.9}, we deduce that if $\lam\in [\lam_*-\eta_m,\lam_*)$, the set $\mathrm{B}_2[R_m,\8]$ is positively invariant for the semiflow $\phi$ generated by \eqref{4.8}. 

On the other hand, let $\lam<\lam_*$. As the norm $\|\cdot\|_{L^1(\mathbb{R}^N)}$ of $L^1(\R^N)$ and that of $X=L^2(\mathbb{R}^N)$ are equivalent in $Y^2,$  one has
\begin{align}\label{4.16}
\big(\~f(w+\xi_\lam(w)),w\big)&=\int_{\R^N} f(x,w+\xi_\lam(w))w\d x\leq f_\8\|w\|_{L^1(\R^N)}\leq C_1|w|\nonumber\\
&\leq \frac{\lam_*-\lam}{2}|w|^2+\frac{1}{2(\lam_*-\lam)}C_1^2,
\end{align}
where $C_1>0$ depends on $f_\8$.
Combining \eqref{4.11} and \eqref{4.16}, it yields
$$
\frac{\d}{\d t}|w(t)|^2\leq (\lam-\lam_*)|w|^2+\frac{C_1^2}{\lam_*-\lam}.
$$
Hence it follows from the Gronwall's inequality that
\be\label{4.17}
  |w(t)|^2\leq \e^{-(\lam_*-\lam)t}|w(0)|^2+\big(1-\e^{-(\lam_*-\lam)t}\big)\frac{C_1^2}{(\lam_*-\lam)^2},\hs t\geq 0.
\ee
Set $$R_\lam=\frac{C_1}{\lam_*-\lam}.$$ Then one can deduce from \eqref{4.17} that
the set $\{w\in Y^2, |w|\leq R_\lam\}$ is positively invariant for $\phi$.

Recalling that $\eta_m\ra 0$ as $m\ra \8$, one can deduce that $$[\lam_*-\eta_1,\lam_*)=\Cup_{m\geq 1}[\lam_*-\eta_m,\lam_*-\eta_{m+1}).$$ We pick $\theta_1=\eta_1$. Then for each $\lam\in [\lam_*-\theta_1,\lam_*)$, there exists some $m\geq 1$ such that $\lam\in [\lam_*-\eta_m,\lam_*-\eta_{m+1})$. It is easy to see that $$R_m\ra \8\Hs \mb{as \hs $\lam\ra\lam_*^-$}.$$ Hence if we choose $r_\lam=R_m$, then one can conclude from the above argument that $\mathrm{B}_2[r_\lam,R_\lam]$ is the desired positively invariant set of the semiflow $\phi$ for $\lam\in [\lam_*-\theta_1,\lam_*)$. The proof of this lemma is complete.
\end{proof}
\subsection{dynamic bifurcation from infinity}
Now we study the bifurcation from infinity of \eqref{4.3}-\eqref{4.4} near $\lam_*$.
\begin{defn}\label{defn2.1}\cite{LLZ}
We say that the system \eqref{4.3}-\eqref{4.4} bifurcates from infinity at $\lam=\lam_*$
(or, $(\infty,\lam_*)$ is a bifurcation point), if for every $\varepsilon>0$, there exist $\lam\in \R$ with $|\lam-\lam_*|<\ve$ and a bounded full solution $u_\lam=u_{\lam}(t)$ of \eqref{4.3}-\eqref{4.4} such that  $$\|u_\lam\|_\8> 1/\ve,$$
where $\|u_\lam\|_\8=\sup_{t\in\R}\|u_{\lam}(t)\|$.
\end{defn}
Let $U\subset Y^2$ and denote $K_\8(\phi_\lam,U)$ the union of all bounded full orbits in $U$.  For simplicity, if $U=Y^2$, we write $K_\8(\phi_\lam,Y^2)=K_\8(\phi_\lam)$. Given a set $K\ss Y^2\times J$ and $\lam\in J$, write $$K[\lam]=\{x:(x,\lam)\in Y^2\times J\}.$$
$K[\lam]$ is called the {\it $\lam$-section} of $K$.
\vs
In the following we first give a result on bifurcations from infinity of the reduced system \eqref{4.8}.
\begin{thm}\label{t4.2}
Let the assumptions {\rm ($\mathbf{A1}$)-($\mathbf{A3}$)} and {\rm ({\bf F})} hold true, and $\lam_*\in \sig_p$. Then $(\infty,\lam_*)$ is a bifurcation point of \eqref{4.8}. Specifically,
there exist a sequence $\lam_n$ with $\lam_n\ra\lam_*$ $(n\ra\8)$ and a bounded full solution $u_n=u_{\lam_n}(t)$ of \eqref{4.8} such that  $$\|u_n\|_\8\ra \8,\hs \mb{as}\hs n\ra\8.$$ Moreover, there exists a neighborhood $\Lam$ of $\lam_*$ such that for each $\lam\in\Lam\setminus\{\lam_*\}$, \eqref{4.8} has a bounded full solution $v_\lam$,
which remains bounded on $\Lam$.
\end{thm}

\begin{proof}
Consider the system
\be\label{4.19}
w_t+A^2 w=P_2\tilde{f}(w+\xi_\lam(w)),\hs w\in Y^2,
\ee
where $A^2=P_2A_\lam=P_2(A-\lam I)$, $\lam\in J=(\lam_*-\frac{\beta}{4},\lam_*+\frac{\beta}{4})$.
Let us study the following linear equation
\be\label{e4.19}
\frac{\d w}{\d t}+A^2w=0. \ee
Pick two numbers $\ve_1,\ve_2$ with $\ve_1\in (\lam_*-\frac{\beta}{4},\lam_*)$ and $\ve_2\in (\lam_*,\lam_*+\frac{\beta}{4})$, respectively. Then if $\lam=\ve_1,\ve_2$, the trivial solution \{0\} is an isolated invariant set for the semiflow $\psi_\lam$ generated by equation \eqref{e4.19} in $Y^2$. By \cite{Ry} (see
Chapter I, Corollary 11.2) we know that there is a positive
integer $p$ such that \be\label{4.20}
h(\psi_{\ve_1},\{0\})=\Sigma^0, \hs h(\psi_{\ve_2},\{0\})=\Sigma^p.\ee ($p$ is
actually the algebraic multiplicity of the eigenvalue
$\lam_*$ of $A^2$.)

Now we consider the nonlinear equation  \be\label{4.22}
\frac{\d w}{\d t}+A^2w-\nu P_2\tilde{f}(w+\xi_\lam(w))=0, \ee where $\nu\in[0,1]$ is the homotopy parameter. By appropriately modifying the argument in the proof of \cite[Chapter II, Theorem 5.1]{Ry} (see also the proof of \cite[Theorem 3.2]{W2}), one can verify that for each $\ve>0$ with $$\ve_1<\lam_*-\ve<\lam_*+\ve<\ve_2,$$ there exists $R_\ve>0$ such that if $u_\lam(t)$ is a bounded full solution of \eqref{4.22} for $\lam\in[\ve_1,\lam_*-\ve]\cup[\lam_*+\ve,\ve_2]$ and $\nu\in [0,1]$, then
\be\label{4.23}\|u_\lam\|_\8<R_\ve.\ee
Denote $\psi_\lam^\nu$  the semiflow generated by the equation \eqref{4.22}.
By the continuation property of the Conley index, we deduce that
\be\label{4.24}\ba{ll}
 h\(\phi_\lam,K_\8({\phi_\lam})\)&= h\(\psi_\lam^1,K_\8({\psi^1_\lam})\)\\[1ex]&=h\(\psi_\lam^0,K_\8({\psi^0_\lam})\)=h(\psi_{\ve_1},\{0\})=\Sigma^0\ea
 \ee for $\lam\in[\ve_1,\lam_*-\ve]$ and
 \be\label{4.25}\ba{ll}
h\(\phi_\lam,K_\8({\phi_\lam})\)&= h\(\psi_\lam^1,K_\8({\psi^1_\lam})\)\\[1ex]&=h\(\psi_\lam^0,K_\8({\psi^0_\lam})\)=h(\psi_{\ve_2},\{0\})=\Sigma^p\ea
 \ee for $\lam\in[\lam_*+\ve,\ve_2]$.
Moreover, we infer from \eqref{4.23} that for the $\ve>0$,
$$
  K_\8({\phi_\lam})[\lam]\ss {\rm B}_{Y^2}(R_\ve),\hs \lam \in[\ve_1,\lam_*-\ve]\cup[\lam_*+\ve,\ve_2],
  $$
where ${\rm B}_{Y^2}(R_\ve)$ denotes a ball in $Y^2$ centered at $0$ with radius $R_\ve$.
Hence, we conclude from \cite[Theorem 3.5]{LLZ} that
there is a sequence $\lam_n\ra \lam_* (n\ra\8)$ and a sequence of bounded full solutions $u_n=u_{\lam_n}(t)$ of \eqref{4.8} for $\lam=\lam_n$ such that \be\label{}\|u_n\|_\8\ra \8,\hs n\ra\8.\ee

On the other hand, we infer from Lemma \ref{l4.2} that there are $R_0,c_0>0$ such that for each $\lam\in[\lam_*,\lam_*+\frac{1}{4}\beta)$, if $w(t)$ is a bounded full solution of \eqref{4.8} in $\mathrm{B}_{2}[R_0,\8]$, then \eqref{4.9} holds true. Therefore
\be\label{4.27}K_\8(\phi_\lam)\subset \mathrm{B}_2[0,R_0], \hs \lam\in[\lam_*,\lam_*+\frac{1}{4}\beta).\ee
Since $\ve_1\in (\lam_*-\frac{\beta}{4},\lam_*)$, $\ve_2\in (\lam_*,\lam_*+\frac{\beta}{4})$ and $\ve$ can be arbitrary, we deduce from \eqref{4.24} and \eqref{4.25} that
\be\label{e5.34}
 h(\phi_\lam,K_\8(\phi_\lam))= \left\{\ba{ll}\Sigma^0,\Hs\lam\in(\lam_*-\frac{\beta}{4},\lam_*);\\[1ex]
 \Sigma^{p},\Hs \lam\in(\lam_*,\lam_*+\frac{\beta}{4}).\ea\right.
 \ee
Thus it follows from \eqref{4.27} that $$h(\phi_{\lam_*},K_\8(\phi_{\lam_*}))=\Sigma^{p}.$$
Now we pick a bounded isolating neighborhood $N$ of $K_\8(\phi_{\lam_*})$. Then there exists $\de>0$ such that $N$ is an isolating neighborhood of $\phi_\lam$ for $\lam\in [\lam_*-\de,\lam_*+\de]$. Thus $$h(\phi_\lam,K_\lam)=\Sigma^{p},\hs \lam\in [\lam_*-\de,\lam_*+\de].$$
where $K_\lam=K_\8(\phi_\lam,N)$. Choose a full solution $v_\lam$ in $K_\lam$ and write $\Lam=[\lam_*-\de,\lam_*+\de]$. Then $\|v_\lam\|_\8$ remains bounded on $\Lam$.
\end{proof}

By virtue of Lemma \ref{l4.2}, we can present our dynamic bifurcation results from infinity for the system \eqref{4.3}-\eqref{4.4}.
\begin{thm}\label{t4.3}
Let the assumptions {\rm ($\mathbf{A1}$)-($\mathbf{A3}$)} hold. Assume the Lipschitz constant $L_f$ of $f$ satisfies {\rm ({\bf F})}. 
Then there exists $\theta_1\in (0,\frac{1}{4}\beta)$ such that for each $\lam\in \Lam_1:=[\lam_*-\theta_1,\lam_*)$, the system $\Phi_\lam$ has a compact invariant set $\cS^\8_\lam$, which takes the shape of $(m-1)$-dimensional topological sphere $\mathbb{S}^{m-1}$, and
\be\label{4.19}
  \lim_{\lam\ra \lam_*^-}\min\{\|w\|: w\in \cS^\8_\lam\}=\8.
\ee
Moreover, there exists a two-sided neighborhood $\Lam$ of $\lam_*$ such that for each $\lam\in \Lam$, $\Phi_\lam$ has a compact invariant set $\cK_\lam$, which remains bounded on $\Lam$.
\end{thm}

\begin{proof}
Let $\theta_1$  be given in Lemma \ref{l4.2}. By virtue of Lemma \ref{l4.2} (3), we see that for each  $\lam\in \Lam_1$, the semiflow $\phi_\lam$ generated by \eqref{4.8} has a positively invariant set ${\rm B}_\lam:={\rm B}_2[r_\lam,R_\lam]$. Note that $Y_2$ is finite dimensional. Then one can conclude from the attractor theory (see e.g., \cite{Chep,Tem}) that the set $$\cA_\lam^\8=\Cap_{s\geq 0}\ol{\Cup_{t\geq s}\phi_\lam(t){\rm B}_\lam}$$  is a global attractor of the semiflow $\phi_\lam$ restricted on ${\rm B}_\lam$.

By the basic knowledge on algebraic topology, it can be easily seen that ${\rm B}_\lam$ has the homotopy type of an $(m-1)$-dimensional topological sphere. Thus ${\rm B}_\lam$ enjoys the same shape of $\mathbb{S}^{m-1}$. Thanks to the shape theory of attractors (see \cite[Theorem 3.6]{KR} or \cite{San,WLD}), we conclude that $\cA_\lam^\8\ss {\rm B}_\lam$ has the shape of $\mathbb{S}^{m-1}$.

Set
$$
  \cS_\lam^\8=\{w+\xi_\lam(w):w\in \cA_\lam^\8\}.
$$
Then $\cS_\lam^\infty\ss\cM_\lam$ is a compact invariant set of the system $\Phi_\lam$ generated by \eqref{4.3}-\eqref{4.4}, which takes the shape of an $(m-1)$-dimensional sphere. By the construction of ${\rm B}_\lam$, we see that $r_\lam,R_\lam\ra \8$ as $\lam\ra\lam_*^-$,  from which one can conclude that \eqref{4.19} holds true.

Let $\Lam$ be obtained in Theorem \ref{t4.2}. Then we infer from the proof of Theorem \ref{t4.2} that for each $\lam\in \Lam$, $\phi_\lam$ has a compact invariant set $K_\lam$, which remains bounded on $\Lam$.
Similarly, set $$\cK_\lam=\{w+\xi_\lam(w):w\in K_\lam\}.$$
Clearly, $\cK_\lam$ is the desired compact invariant set of $\Phi_\lam$ for $\lam\in \Lam$. The proof of the theorem is complete.
\end{proof}

As a consequence of Theorem \ref{t4.3}, we can obtain our main results on bifurcations from infinity and multiplicity of the Schr\"odinger equation \eqref{e1.1}.

\begin{thm}
Let the assumptions {\rm ($\mathbf{A1}$)}-{\rm ($\mathbf{A3}$)} hold and $\lam_*\in \sig_p$. Assume the Lipschitz constant $L_f$ of $f$ satisfies {\rm ({\bf F})}. Then there exists $\theta\in(0,\frac{1}{4}\beta)$ such that for each $\lam\in \Lam_2:=[\lam_*-\theta,\lam_*),$
the Schr\"odinger equation \eqref{e1.1} has at least three distinct nontrival solutions $e_\lam^1,e_\lam^2,e^3_\lam$ with
$$
  \|e^{i}_\lam\|\ra \infty,\hs \mb{as \,\,$\lam\ra \lam_*^-$,\hs $i=1,2$,}
$$
whereas $e_\lam^3$ remains bounded on $\Lam_2$.
\begin{proof}
Let $\theta_1$ and $\Lam$ be given in Theorem \ref{t4.3}. Pick a positive number $\theta$ such that $\theta<\min\{\theta_1,\de\}$, where $\de>0$ is chosen in the proof of Theorem \ref{t4.2}.
By virtue of Theorem \ref{t4.3}, we deduce that for each $\lam\in[\lam_*-\theta,\lam_*)$, the system $\Phi_\lam$ has a compact invariant set $\cS^\8_\lam$.
In what follows we show that $\cS^\8_\lam$ has at least two distinct nontrivial stationary solutions $e^1_\lam$ and $e^2_\lam$, which are actually the solutions of \eqref{e1.1}.

Observing that $\cS^\8_\lam$ has the shape of $\mathbb{S}^{m-1}$ and $m\geq 1$, we deduce that $\cS^\8_\lam$ has at least two distinct points $u$ and $v$. If $u$ and $v$ are equilibria, then the result holds true. Now we suppose that $u$ is not an equilibrium. Then there exists a full solution $\gam=\gam(t)$ of $\Phi_\lam$ in $\cS_\lam^\8$ such that $\gam(0)=u$. Since $\Phi_\lam$ is a gradient system, we conclude that $$\w(\gam)\cap\w^*(\gam)=\emptyset.$$ Noticing that $\w(\gam)$ and $\w^*(\gam)$ consist of equilibrium points, one can see that $\Phi_\lam$ has two distinct equilibria in $\cS^\8_\lam$. Recalling that $\cS^\8_\lam$ satisfies \eqref{4.19}, we conclude that the equation \eqref{e1.1} has two distinct nontrivial solutions $e^1_\lam$ and $e^2_\lam$ with
$$
  \|e^{i}_\lam\|\ra \infty,\hs \mb{as \,$\lam\ra \lam_*^-$,\hs $i=1,2$.}
$$
Theorem \ref{t4.3} also asserts that for each $\lam\in \Lam=[\lam_*-\de,\lam_*+\de]$, the system $\Phi_\lam$ has a compact invariant set $\cK_\lam$, which remains bounded on $\Lam$. Thus $\Phi_\lam$ has at least one equilibrium $e_\lam^3$ in $\cK_\lam$. Therefore $e_\lam^3$ is the solution of the Schr\"odinger equation \eqref{e1.1} fulfilling the requirements of the theorem for $\lam\in \Lam_2$. The proof of the theorem is complete.
\end{proof}
\end{thm}
\br
 The ``dual" version of all our results holds true if, instead of \eqref{1.2}, we assume that
\eqref{1.5} is fulfilled.
\er
\vs
\noindent{\bf Acknowledgments}
\Vs
This work is supported by National Natural Science Foundation of China (Grant Nos. 11871368, 11801190).


\medskip
\medskip

\end{document}